\documentclass[11pt, reqno]{amsart}
\usepackage{amssymb,amsthm, amsmath, amsfonts}
\usepackage{graphics}
\usepackage{hyperref}
\usepackage[all]{xy}
\usepackage{enumerate}
\usepackage{mathrsfs}

\theoremstyle{plain}
\newtheorem{theorem}{Theorem}
\newtheorem{lemma}[theorem]{Lemma}
\newtheorem{proposition}[theorem]{Proposition}
\newtheorem{corollary}[theorem]{Corollary}

\theoremstyle{definition}

\newtheorem{notation}[theorem]{Notation}

\DeclareMathOperator{\reg}{reg}
\newcommand{\low}{{\mathrm{low}}}
\newcommand{\FI}{{\mathrm{FI}}}
\newcommand{\hd}{{\mathrm{hd}}}
\newcommand{\kk}{{\Bbbk}}
\newcommand{\N}{{\mathrm{N}}}
\newcommand{\sst}{{\sqcup\{\star\}}}
\newcommand{\td}{{\mathrm{td}}}
\newcommand{\ws}{{\widetilde{{S}}}}
\newcommand{\Z}{{\mathbb{Z}}}

\title{A long exact sequence for homology of FI-modules}

\begin{document}

\author{Wee Liang Gan}
\address{Department of Mathematics, University of California, Riverside, CA 92521, USA}
\email{wlgan@math.ucr.edu}

\begin{abstract}
We construct a long exact sequence involving the homology of an FI-module. Using the long exact sequence, we give two methods to bound the Castelnuovo-Mumford regularity of an FI-module which is generated and related in finite degree. We also prove that for an FI-module which is generated and related in finite degree, if it has a nonzero higher homology, then its homological degrees are strictly increasing (starting from the first homological degree).
\end{abstract}

\maketitle

\section{Introduction}

This article studies homological aspects of the theory of FI-modules. We begin by recalling a few definitions from \cite{CE}, \cite{CEF}, and \cite{CEFN}.

Let $\Z_+$ be the set of non-negative integers. Let $\kk$ be a commutative ring. Let $\FI$ be the category whose objects are the finite sets and whose morphisms are the injective maps. An \emph{$\FI$-module} is a functor from $\FI$ to the category of $\kk$-modules. For any $\FI$-module $V$ and finite set $X$, we shall write $V_X$ for $V(X)$.

Suppose $V$ is an $\FI$-module. For any finite set $X$, let $(JV)_X$ be the $\kk$-submodule of $V_X$ spanned by the images of the maps $f_*: V_Y \to V_X$ for all injections $f:Y \to X$ with $|Y|<|X|$. Then $JV$ is an $\FI$-submodule of $V$. Let $$F(V):=V/JV.$$ Then $F$ is a right exact functor from the category of $\FI$-modules to itself. Following \cite{CE} and \cite{CEF}, for any $a\in \Z_+$, the \emph{FI-homology functor} $H_a$ is defined to be the $a$-th left derived functor of $F$. 

Fix a one-element set $\{\star\}$ and define a functor $\sigma: \FI \to \FI$ by $X \mapsto X \sst$. If $f:X\to Y$ is a morphism in $\FI$, then $\sigma(f): X\sst \to Y \sst$ is the map $f \sqcup \mathrm{id}_{\{\star\}}$. Following \cite[Definition 2.8]{CEFN}, the \emph{shift functor} $S$ from the category of $\FI$-modules to itself is defined by $SV = V \circ \sigma$ for every $\FI$-module $V$.

Suppose $V$ is an $\FI$-module. For any finite set $X$, one has $(SV)_X = V_{X\sst}$. There is a natural $\FI$-module homomorphism 
\begin{equation*}
\iota: V \to SV
\end{equation*}
where the maps $\iota_X : V_X \to (SV)_X$ are defined by the inclusion maps $X \hookrightarrow X\sst$. We denote by $DV$ the cokernel of $\iota: V \to SV$. Following \cite[Definition 3.2]{CE}, we call the functor $D: V \mapsto DV$ the \emph{derivative functor} on the category of $\FI$-modules.

Our main result is the following.

\begin{theorem} \label{main theorem}
Let $V$ be an $\FI$-module. Then there is a long exact sequence:
\begin{equation*}
\xymatrix{
& \quad\cdots\cdots\quad \ar[r] & SH_{a+1}(V) \ar[dll] & \\
H_a(V) \ar[r]^{\iota_*} & H_a(SV) \ar[r] & SH_a(V) \ar[dll] & \\
H_{a-1}(V) \ar[r]^{\iota_*} & \quad\cdots\cdots\quad \ar[r] & SH_0(V) \ar[r] & 0.
}
\end{equation*}
\end{theorem}

The proof of Theorem \ref{main theorem} will be given in Section \ref{les section}.

As applications of Theorem \ref{main theorem}, we give in Section \ref{C-M reg section} two methods to bound from above the Castelnuovo-Mumford regularity of an $\FI$-module which is generated and related in finite degree. The first method, using the derivative functor $D$, gives a new proof of the bound first found by Church and Ellenberg \cite[Theorem A]{CE}. The second method, using the shift functor $S$, gives a bound which is always less than or equal to one found recently by Li and Ramos \cite[Theorem 5.20]{LRa}. Along the way, we use Theorem \ref{main theorem} to reprove a few results of Li and Yu \cite{LYu}, and Ramos \cite{Ra}. We also prove that for an FI-module which is generated and related in finite degree, if it has a nonzero higher homology, then its homological degrees are strictly increasing (starting from the first homological degree).

Although some of the results in Section \ref{C-M reg section} are known, our proofs based on Theorem \ref{main theorem} seem to be more direct than previous proofs.

\subsection*{Acknowledgment}
I thank Liping Li for useful discussions.

\section{The long exact sequence} \label{les section}

\subsection{A Koszul complex}
Let $V$ be an $\FI$-module. The $\FI$-homology of $V$ can be computed from a Koszul complex $\ws_{-\bullet}V$ first defined in \cite[(11)]{CEFN}. We recall the construction of this complex following \cite[Section 2]{GL}.

For any finite set $I$, let $\kk I$ be the free $\kk$-module with basis $I$, and $\det(I)$ the free $\kk$-module $\bigwedge^{|I|} \kk I$ of rank one; by convention, if $I=\emptyset$, then $\det(I)=\kk$. If $I=\{i_1,\ldots,i_a\}$, then $i_1\wedge\cdots\wedge i_a$ is a basis for $\det(I)$. 

Suppose $X$ is a finite set and $Y$ is a subset of $X$. If $i\in X\setminus Y$ and $v\in V_Y$, we shall write $i(v)$ for the element $f_*(v)\in V_{Y\cup\{i\}}$ where $f:Y \hookrightarrow Y\cup\{i\}$ is the inclusion map. For any $a\in \Z_+$, let
\begin{equation*}
(\ws_{-a} V)_X := \bigoplus_{\substack{I\subset X \\ |I|=a}} V_{X\setminus I} \otimes_{\kk} \det(I).
\end{equation*}
The differential $d : (\ws_{-a} V)_X \to (\ws_{-a+1} V)_X$ is defined on each direct summand by the formula
\begin{equation*}
d( v\otimes i_1 \wedge \cdots \wedge i_a) := \sum_{p=1}^a (-1)^p \, i_p(v) \otimes i_1 \wedge \cdots \widehat{i_p} \cdots \wedge i_a,
\end{equation*}
where $v\in V_{X\setminus I}$, $I=\{i_1,\ldots, i_a\}$, and $\widehat{i_p}$ means that $i_p$ is omitted in the wedge product. 

Suppose $X$ and $X'$ are finite sets and $f:X\to X'$ is an injective map. For any $I\subset X$, the map $f$ restricts to an injective map $f\mid_{X\setminus I} : X\setminus I \to X'\setminus f(I)$. We define 
\begin{equation*}
f_*: (\ws_{-a}V)_X \to (\ws_{-a}V)_{X'}
\end{equation*}
by the formula
\begin{equation*}
f_*( v\otimes i_1 \wedge\cdots\wedge i_a) := (f\mid_{X\setminus I})_*(v) \otimes f(i_1)\wedge \cdots \wedge f(i_a),
\end{equation*}
where $v\in V_{X\setminus I}$, $(f\mid_{X\setminus I})_*(v)\in V_{X'\setminus f(I)}$, and $I=\{i_1,\ldots,i_a\}$. This defines, for each $a\in \Z_+$, an $\FI$-module $\ws_{-a} V$. 

We obtain a complex $\ws_{-\bullet}V$ of $\FI$-modules:
\begin{equation*}
\cdots \longrightarrow \ws_{-2}V \longrightarrow \ws_{-1}V \longrightarrow \ws_0 V \longrightarrow 0.
\end{equation*}
The following theorem was independently proved in \cite{CE} and \cite{GL}.

\begin{theorem} \label{koszul complex computes homology}
Let $V$ be an $\FI$-module. Then there is an $\FI$-module isomorphism 
\begin{equation*}
H_a(V) \cong H_a(\ws_{-\bullet}V) \quad \mbox{ for each } a\in \Z_+.
\end{equation*}
\end{theorem}
\begin{proof}
See \cite[Proposition 4.9 and proof of Theorem B]{CE}, or \cite[Theorem 1 and Remark 4]{GL}.
\end{proof}

Applying the shift functor $S$ to the complex $\ws_{-\bullet}V$, we obtain a complex $S\ws_{-\bullet}V$. Since $S$ is an exact functor, it is immediate from Theorem \ref{koszul complex computes homology} that one has an isomorphism 
\begin{equation*}
SH_a(V) \cong H_a( S \ws_{-\bullet} V ) \quad \mbox{ for each } a\in \Z_+.
\end{equation*}

\subsection{Proof of Theorem \ref{main theorem}}
Let $V$ be an $\FI$-module. The homomorphism $\iota: V \to SV$ defines, in the obvious way, a morphism of complexes $\widetilde\iota: \ws_{-\bullet}V \to \ws_{-\bullet}SV$. By a standard result in homological algebra \cite[Section 1.5]{We}, Theorem \ref{main theorem} is immediate from Theorem \ref{koszul complex computes homology} and the following lemma.

\begin{lemma}  \label{mapping cone lemma}
Let $V$ be an $\FI$-module. Then the complex $S\ws_{-\bullet}V$ is isomorphic to the mapping cone of $\widetilde\iota: \ws_{-\bullet}V \to \ws_{-\bullet}SV$.
\end{lemma}
\begin{proof}
We shall follow standard notations (found, for example, in \cite[Section 1.5]{We}) and write the mapping cone of $\widetilde\iota$ as
\begin{equation*}
\mathrm{cone}(\widetilde\iota) = (\ws_{-\bullet}V)[-1] \oplus \ws_{-\bullet}SV.
\end{equation*} 
To define a homomorphism 
\begin{equation*}
\phi : \mathrm{cone}(\widetilde\iota)  \longrightarrow S\ws_{-\bullet}V,
\end{equation*}
we need to define, for each finite set $X$, a homomorphism of complexes
\begin{equation*}
\phi_X :   \mathrm{cone}(\widetilde\iota)_X  \longrightarrow (S\ws_{-\bullet}V)_X.
\end{equation*}

Suppose $a>0$. The degree $a$ component of $(\ws_{-\bullet}V)_X[-1]$ is $(\ws_{-(a-1)} V)_X$. For any $I=\{i_1,\ldots, i_{a-1}\} \subset X$ and $v\in V_{X\setminus I}$, we have the element $v\otimes i_1\wedge \cdots \wedge i_{a-1} \in (\ws_{-(a-1)} V)_X$; let
\begin{equation*}
\phi_X(v \otimes i_1\wedge\cdots\wedge i_{a-1}) := v \otimes \left( (\star) \wedge i_1\wedge\cdots\wedge i_{a-1} \right) \in (\ws_{-a}V)_{X\sst}.
\end{equation*}
Here, we used $X\setminus I = (X\sst)\setminus (I\sst)$ to see that the element $v$ on the right hand side is an element of $V_{(X\sst) \setminus (I\sst)}$.

Suppose $a\geqslant 0$. For any $I=\{i_1,\ldots, i_a\} \subset X$ and $v\in (SV)_{X\setminus I}$, we have the element $v\otimes i_1\wedge \cdots \wedge i_a \in (\ws_{-a}SV)_X$; let
\begin{equation*}
\phi_X(v \otimes i_1\wedge\cdots\wedge i_a) := v \otimes i_1\wedge\cdots\wedge i_a \in (\ws_{-a}V)_{X\sst}.
\end{equation*}
Here, we used $(SV)_{X\setminus I} = V_{(X\sst) \setminus I}$ to see that the element $v$ on the right hand side is an element of $V_{(X\sst) \setminus I}$.

By a routine verification, the above defines a homomorphism $\phi$ of complexes of $\FI$-modules. It is plain that $\phi$ is bijective and hence an isomorphism.
\end{proof}

A special case of Lemma \ref{mapping cone lemma} appeared in \cite[proof of Proposition 6]{GL}.

\section{Applications} \label{C-M reg section}

\subsection{Definitions and notations}
We recall some definitions from \cite{CE} and \cite{Li1}.

Let $V$ be any $\FI$-module. For any $n\in \Z_+$, we set $\mathbf{n} :=\{1,\ldots,n\}$ (in particular, $\mathbf{0}=\emptyset$). We shall use the convention that the supremum and infimum of an empty set are $-\infty$ and $\infty$, respectively.

The \emph{degree} $\deg(V)$ of $V$ is 
\begin{equation*}
\deg(V) := \sup\{ n\in\Z_+ \mid V_{\mathbf{n}} \neq 0 \}.
\end{equation*}

The \emph{lowest degree} $\low(V)$ of $V$ is
\begin{equation*}
\low(V) := \inf\{ n\in\Z_+ \mid V_{\mathbf{n}} \neq 0 \}.
\end{equation*}

For any $a\in \Z_+$, the \emph{$a$-th homological degree} $\hd_a(V)$ of $V$ is
\begin{equation*}
\hd_a(V) := \deg H_a(V).
\end{equation*}

The \emph{Castelnuovo-Mumford regularity} $\reg(V)$ of $V$ is the infimum of the set of all $c\in\Z$ such that
\begin{equation*}
\hd_a(V) \leqslant c+a \quad \mbox{ for every integer } a\geqslant 1. 
\end{equation*}

The \emph{torsion degree} $\td(V)$ of $V$ is the supremum of the set of all $n\in\Z_+$ such that there exists a nonzero $v\in V_\mathbf{n}$ satisfying $f_*(v)=0$ for every injection $f:\mathbf{n} \to \mathbf{n+1}$.

For any $k\in \Z_+$, we say that $V$ is \emph{generated in degree $\leqslant k$} if $\hd_0(V)\leqslant k$.

For any $k, d\in \Z_+$, we say that $V$ is \emph{generated in degree $\leqslant k$ and related in degree $\leqslant d$} if there exists a short exact sequence 
\begin{equation*}
0 \longrightarrow W \longrightarrow P \longrightarrow V \longrightarrow 0
\end{equation*}
where $P$ is a projective $\FI$-module generated in degree $\leqslant k$ and $W$ is an $\FI$-module generated in degree $\leqslant d$. 

Let $KV$ be the kernel of $\iota: V\to SV$.

Let $H_1^{D}$ be the first left-derived functor of the right exact functor $D$.

\subsection{Some basic facts}
We collect in the following lemma some basic facts which we shall use later.

\begin{lemma} \label{lemma basic facts}
Let $V$ be an $\FI$-module. Then one has the followings.

\begin{itemize}
\item[(i)] There is an isomorphism $KV \cong H_0(KV)$, and $\td(V)=\deg(KV)$.

\item[(ii)] There is an isomorphism $KV \cong H_1^D(V)$.

\item[(iii)] If $P$ is a projective $\FI$-module, then $DP$ is a projective $\FI$-module.

\item[(iv)] If $V$ is generated in degree $\leqslant k$ where $k\in \Z_+$, then $DV$ is generated in degree $\leqslant k-1$.

\item[(v)] If $V$ is generated in degree $\leqslant k$ and related in degree $\leqslant d$ where $k, d\in \Z_+$, then $DV$ is generated in degree $\leqslant k-1$ and related in degree $\leqslant d-1$.

\item[(vi)]  If $V$ is generated in degree $\leqslant k$ and related in degree $\leqslant d$ where $k, d\in \Z_+$, then $\hd_1(V)\leqslant d$.

\item[(vii)] There is an isomorphism $S(DV)\cong D(SV)$.

\item[(viii)] If $P$ is a projective $\FI$-module, then $SP$ is a projective $\FI$-module.

\item[(ix)] If $V$ is generated in degree $\leqslant k$, then $SV$ is generated in degree $\leqslant k$.

\item[(x)] If $V$ is generated in degree $\leqslant k$ and related in degree $\leqslant d$ where $k, d\in \Z_+$, then $SV$ is generated in degree $\leqslant k$ and related in degree $\leqslant d$.

\end{itemize}
\end{lemma}
\begin{proof}
(i) Trivial.

(ii) See \cite[Lemma 3.6(i)]{CE}.

(iii) See \cite[Lemma 3.6(iv)]{CE}.

(iv) See \cite[proof of Proposition 3.5]{CE}.

(v) Follows from (iii) and (iv).

(vi) Trivial.

(vii) See \cite[Lemma 3.5]{Ra}.

(viii) Follows from \cite[Proposition 2.12]{CEFN}.

(ix) See \cite[Corollary 2.13]{CEFN}.

(x) Follows from (viii) and (ix).
\end{proof}

The following simple observation is sometimes useful.

\begin{lemma} \label{lemma on lowest degree}
Let $V$ be an $\FI$-module and let $a\in\Z_+$. If $n<\low(V)+a$, then $H_a(V)_{\mathbf{n}} = 0$.
\end{lemma}
\begin{proof}
If $n<\low(V)+a$, then $(\ws_{-a}V)_{\mathbf{n}}=0$; hence, by Theorem \ref{koszul complex computes homology}, one has $H_a(V)_{\mathbf{n}} = 0$.
\end{proof}

\begin{corollary} \label{lower bound of hd}
Let $V$ be an $\FI$-module and let $a\in\Z_+$. If $H_a(V)\neq 0$, then 
\begin{equation*}
\hd_a(V)\geqslant \low(V)+a.
\end{equation*}
\end{corollary}
\begin{proof}
Immediate from Lemma \ref{lemma on lowest degree}.
\end{proof}

\subsection{FI-modules of finite degree}
The following result was independently proved by Li \cite[Theorem 4.8]{Li1} and Ramos \cite[Corollary 3.11]{Ra}. Let us give a proof using Theorem \ref{koszul complex computes homology}.

\begin{lemma} \label{reg of torsion module}
Let $V$ be an $\FI$-module with $\deg(V)<\infty$. Then $\reg(V) \leqslant \deg(V)$.
\end{lemma}
\begin{proof}
Let $a\in \Z_+$. If $n>\deg(V)+a$, then $(\ws_{-a}V)_{\mathbf{n}}=0$; hence, by Theorem \ref{koszul complex computes homology}, one has $H_a(V)_{\mathbf{n}}=0$.
\end{proof}

\subsection{Bounding regularity using the derivative functor}
Our first strategy for bounding the Castelnuovo-Mumford regularity $\reg(V)$ of an $\FI$-module $V$ is to find a bound of $\reg(V)$ in terms of $\reg(DV)$, and then use recurrence to obtain a bound for $\reg(V)$.

\begin{proposition} \label{proposition on bound of reg by D}
Let $V$ be an $\FI$-module. Then one has:
\begin{equation*}
\reg(V) \leqslant \max\{ \hd_1(V)-1,\, \td(V),\, \reg(DV)+1 \}.
\end{equation*}
\end{proposition}
\begin{proof}
Set $c = \max\{ \hd_1(V)-1,\, \td(V),\, \reg(DV)+1 \}$. There is nothing to prove if $c=\infty$, so assume $c<\infty$. We need to prove that
\begin{equation} \label{bound of hd by D}
\hd_a(V) \leqslant c + a \quad \mbox{ for each } a\geqslant 1.
\end{equation}
When $a=1$, the inequality (\ref{bound of hd by D}) holds because $\hd_1(V)-1\leqslant c$.

Suppose, for induction on $a$, that one has $\hd_{a-1}(V) \leqslant c+a-1$ for some $a\geqslant 2$. Then 
\begin{equation*}
H_{a-1}(V)_{\mathbf{n}} = 0 \quad \mbox{ for each } n\geqslant c+a.
\end{equation*}

We have two short exact sequences:
\begin{gather*}
0 \longrightarrow KV \longrightarrow V \stackrel{\iota_1}{\longrightarrow} V/KV \longrightarrow 0, \\
0 \longrightarrow V/KV \stackrel{\iota_2}{\longrightarrow} SV \longrightarrow DV \longrightarrow 0. 
\end{gather*}
They give two long exact sequences:
\begin{gather*}
\cdots \longrightarrow H_a(V)_{\mathbf{n}} \stackrel{\iota_{1*}}{\longrightarrow} H_a(V/KV)_{\mathbf{n}} \longrightarrow H_{a-1}(KV)_{\mathbf{n}} \longrightarrow \cdots,\\
\cdots \longrightarrow H_a(V/KV)_{\mathbf{n}} \stackrel{\iota_{2*}}{\longrightarrow} H_a(SV)_{\mathbf{n}} \longrightarrow H_a(DV)_{\mathbf{n}} \longrightarrow \cdots.
\end{gather*}
Recall that $\deg(KV)=\td(V)$ (see Lemma \ref{lemma basic facts}(i)). By Lemma \ref{reg of torsion module} and the inequality $c\geqslant \td(V)$, one has 
\begin{equation*}
H_{a-1}(KV)_{\mathbf{n}}=0 \quad \mbox{ for each } n\geqslant c+a. 
\end{equation*}
By the inequality $c\geqslant \reg(DV)+1$, one has 
\begin{equation*}
H_a(DV)_{\mathbf{n}}=0 \quad \mbox{ for each } n\geqslant c+a. 
\end{equation*}
Since $\iota: V\to SV$ is the composition of $\iota_1: V\to V/K$ and $\iota_2: V/K\to SV$, the map $\iota_*: H_a(V)_{\mathbf{n}} \to H_a(SV)_{\mathbf{n}}$ is surjective for each $n\geqslant c+a$.

From Theorem \ref{main theorem}, we have an exact sequence:
\begin{equation*}
\cdots  \longrightarrow  H_a(V)_{\mathbf{n}} \stackrel{\iota_{*}}{\longrightarrow} H_a(SV)_{\mathbf{n}} \longrightarrow H_a(V)_{\mathbf{n+1}} \longrightarrow H_{a-1}(V)_{\mathbf{n}} \longrightarrow \cdots.
\end{equation*} 
It follows that $H_a(V)_{\mathbf{n+1}}=0$ for $n\geqslant c+a$, and hence $\hd_a(V)\leqslant c+a$.
\end{proof}

Finiteness of the Castelnuovo-Mumford regularity for finitely generated $\FI$-modules over a field of characteristic zero was first proved by Sam and Snowden in \cite[Corollary 6.3.5]{SS}. In the following theorem, the inequalities (\ref{bound on td}) and (\ref{church-ellenberg bound}) were first proved by Church and Ellenberg in \cite[Theorem 3.8 and Theorem 3.9]{CE} via an intricate combinatorial result \cite[Theorem D]{CE}. An alternative proof of (\ref{bound on td}) and (\ref{church-ellenberg bound}) was subsequently given by Li in \cite[Theorem 2.4]{Li3} using results from \cite{Li1} and \cite{LYu}. (Although the papers \cite{Li1}, \cite{Li3}, and \cite{LYu} worked with finitely generated $\FI$-modules over a noetherian ring, most of the arguments in there can be adapted to our more general setting.) The proof of (\ref{bound on td}) we give below follows along the same lines as the argument in \cite{Li3}, and we use the crucial idea in \cite{Li3} of proving the inequalities (\ref{bound on td}) and (\ref{church-ellenberg bound}) simultaneously by induction on $k$. However, our proof of (\ref{church-ellenberg bound}) via (\ref{1st bound}) and (\ref{1st bound vs church-ellenberg bound}) is quite different from the proofs in \cite{CE} and \cite{Li3}.

\begin{theorem} \label{theorem on bound of reg by D}
Let $V$ be an $\FI$-module which is generated in degree $\leqslant k$ and related in degree $\leqslant d$ where $k, d\in \Z_+$. Let
\begin{align*}
\hd_1^D(V) & := \max\{ \hd_1(D^i V)+i \mid i=0,1,\ldots,k \},\\
\td^D(V) &  := \max\{ \td(D^i V)+i \mid i=0,1,\ldots,k \}.
\end{align*}
Then one has the followings:
\begin{align}
\td(V) & \leqslant \min\{k,d\}+d-1, \label{bound on td}\\
\reg(V) & \leqslant \max\{ \hd_1^D(V)-1,\, \td^D(V) \}, \label{1st bound}\\
\max\{ \hd_1^D(V)-1,\, \td^D(V) \} & \leqslant \min\{k,d\}+d-1. \label{1st bound vs church-ellenberg bound}
\end{align}
In particular, one has:
\begin{equation}
\reg(V) \leqslant \min\{k,d\}+d-1. \label{church-ellenberg bound}
\end{equation}
\end{theorem}
\begin{proof}
If $V=0$, then $\td(V)$, $\reg(V)$, $\hd_1^D(V)$, and $\td^D(V)$ are all equal to $-\infty$, so there is nothing to prove.

Suppose $V\neq 0$. We use induction on $k$. We have a short exact sequence
\begin{equation*}
0 \longrightarrow W \longrightarrow P \longrightarrow V \longrightarrow 0
\end{equation*}
where $P$ is a projective $\FI$-module generated in degree $\leqslant k$ and $W$ is an $\FI$-module generated in degree $\leqslant d$. Since $H_1^D(P)=0$ and $H_1^D(V)=KV$ (see Lemma \ref{lemma basic facts}(ii)), we obtain an exact sequence
\begin{equation*}
0  \longrightarrow KV \longrightarrow DW \longrightarrow DP \longrightarrow DV \longrightarrow 0,
\end{equation*}
which we break up as two short exact sequences:
\begin{gather*}
0 \longrightarrow KV \longrightarrow DW \longrightarrow DW/KV \longrightarrow 0,\\
0 \longrightarrow DW/KV \longrightarrow DP \longrightarrow DV \longrightarrow 0.
\end{gather*}
They give two long exact sequences:
\begin{gather}
\cdots \longrightarrow H_1(DW/KV) \longrightarrow H_0(KV) \longrightarrow H_0(DW) \longrightarrow \cdots, \label{a les 1}\\
\cdots \longrightarrow H_2(DV) \longrightarrow H_1(DW/KV) \longrightarrow 0 \longrightarrow \cdots \label{a les 2},
\end{gather}
where we used Lemma \ref{lemma basic facts}(iii) to see that $H_1(DP)=0$. 

By Lemma \ref{lemma basic facts}(v), the $\FI$-module $DV$ is generated in degree $\leqslant k-1$ and related in degree $\leqslant d-1$. Hence, we have:
\begin{align*}
\td(V) & = \hd_0(KV) & \mbox{(by Lemma \ref{lemma basic facts}(i))} \\
& \leqslant \max\{ \hd_0(DW),\, \hd_1(DW/KV) \}& \mbox{(by (\ref{a les 1}))}\\
& \leqslant \max\{ d-1,\, \hd_2(DV) \}& \mbox{(by Lemma \ref{lemma basic facts}(iv) and (\ref{a les 2}))}\\
& \leqslant \max\{ d-1,\, \reg(DV)+2 \}&\\
& \leqslant \max\{ d-1,\, \min\{k-1, d-1\}+(d-1)-1+2\}& \mbox{(by induction hypothesis)} \\
& \leqslant \min\{k,d\}+d-1. &
\end{align*}
We also have:
\begin{align*}
\reg(V) & \leqslant \max\{ \hd_1(V)-1,\, \td(V),\, \reg(DV)+1 \} & \mbox{(by Proposition \ref{proposition on bound of reg by D})} \\
& \leqslant \max\{ \hd_1^D(V)-1,\, \td^D(V) \} & \mbox{(by induction hypothesis)}.\\
\end{align*}
By Lemma \ref{lemma basic facts}(v) and Lemma \ref{lemma basic facts}(vi), for $i=0,\ldots, k$, we have:
\begin{gather*}
\hd_1(D^i V)+i-1 \leqslant (d-i)+i-1 \leqslant \min\{k,d\}+d-1,\\
\td(D^i V)+i \leqslant \min\{k-i, d-i\}+(d-i)-1 + i \leqslant \min\{k,d\}+d-1,
\end{gather*}
where we used (\ref{bound on td}) for $V$, $DV$, \ldots, and $D^k V$. Hence,
\begin{equation*}
\max\{ \hd_1^D(V)-1,\, \td^D(V) \} \leqslant \min\{k,d\}+d-1.
\end{equation*}
\end{proof}

\subsection{Iterated shifts and vanishing of homology}
Recall that the FI-homology functor $H_a$ is, by definition, the $a$-th left derived functor of $F$. An $\FI$-module $V$ is \emph{$F$-acyclic} if $H_a(V)=0$ for every $a\geqslant 1$.

\begin{lemma} \label{acyclicity from V to SV}
Let $V$ be an $\FI$-module. 
\begin{itemize}
\item[(i)] If $a\in \Z_+$ and $H_a(V)=0$, then $H_a(SV)=0$.

\item[(ii)] If $V$ is $F$-acyclic, then $SV$ is $F$-acyclic.
\end{itemize}
\end{lemma}
\begin{proof}
Immediate from Theorem \ref{main theorem}.
\end{proof}

We say that an $\FI$-module $V$ is \emph{torsion-free} if $\td(V)=-\infty$. By Lemma \ref{lemma basic facts}(i), an $\FI$-module $V$ is torsion-free if and only if $KV=0$.

The following lemma can be deduced from \cite[Theorem 3.5 and Lemma 3.12]{LYu} under some finiteness assumptions. We give a proof using Theorem \ref{main theorem}.

\begin{lemma} \label{DV acyclic implies V acyclic}
Let $V$ be a torsion-free $\FI$-module. If $DV$ is $F$-acyclic, then $V$ is $F$-acyclic.
\end{lemma}
\begin{proof}
Since $KV=0$, there is a short exact sequence $0 \to V \stackrel{\iota}{\to} SV \to DV \to 0$. From the long exact sequence in homology and the $F$-acyclicity of $DV$, we deduce that:
\begin{itemize}
\item[(i)] $\iota_*: H_0(V) \longrightarrow H_0(SV)$ is a monomorphism.

\item[(ii)] $\iota_*: H_a(V) \longrightarrow H_a(SV)$ is an isomorphism for each $a\geqslant 1$.
\end{itemize}

Suppose that $a\geqslant 1$. From (i), (ii), and the long exact sequence in Theorem \ref{main theorem}, we must have $SH_a(V)=0$, so $H_a(V)_{\mathbf{n}}=0$ for each $n\geqslant 1$. By Lemma \ref{lemma on lowest degree}, we have $H_a(V)_{\mathbf{0}}=0$. Therefore, $H_a(V)=0$.
\end{proof}

The following result is proved in \cite[Corollary 3.3]{Li3} and \cite[Corollary 4.11]{Ra}; see also \cite[Theorem A]{Na}. We adapt the argument in \cite[Theorem 3.13]{LYu} using (\ref{bound on td}) and Lemma \ref{DV acyclic implies V acyclic}. 

\begin{theorem} \label{acyclic after shifts}
Let $V$ be an $\FI$-module which is generated in degree $\leqslant k$ and related in degree $\leqslant d$ where $k, d\in \Z_+$. Then $S^i V$ is $F$-acyclic for each $i\geqslant \min\{k,d\}+d$. 
\end{theorem}
\begin{proof}
The statement is trivial if $V=0$.

Suppose $V\neq 0$. We prove the theorem by induction on $k$. Suppose $i\geqslant \min\{k,d\}+d$. By Theorem \ref{theorem on bound of reg by D}, the $\FI$-module $S^i V$ is torsion-free. Using Lemma \ref{lemma basic facts}(v), one has:
\begin{align*}
& \mbox{ $S^i(DV)$ is $F$-acyclic (by induction hypothesis)} \\
\Longrightarrow & \mbox{ $D(S^i V)$ is $F$-acyclic (by Lemma \ref{lemma basic facts}(vii))} \\
\Longrightarrow & \mbox{ $S^i V$ is $F$-acyclic (by Lemma \ref{DV acyclic implies V acyclic}). } 
\end{align*}
\end{proof}

\begin{notation} \label{NV}
If $V$ is an $\FI$-module which is generated in degree $\leqslant k$ and related in degree $\leqslant d$ where $k, d\in \Z_+$, we denote by $\N(V)$ the minimum $i\in\Z_+$ such that $S^i V$ is $F$-acyclic.
\end{notation}

By Theorem \ref{acyclic after shifts}, one has: 
\begin{equation*}
\N(V)\leqslant \min\{k,d\}+d.
\end{equation*}

The following result is proved in \cite[Theorem 1.3]{LYu} and \cite[Theorem B]{Ra}. We give another proof using Theorem \ref{main theorem}.

\begin{proposition} \label{acyclic if one homology is zero}
Let $V$ be an $\FI$-module which is generated in degree $\leqslant k$ and related in degree $\leqslant d$ where $k, d\in \Z_+$. Then $V$ is $F$-acyclic if and only if there exists $s \geqslant 1$ such that $H_s(V)=0$.
\end{proposition}
\begin{proof}
We only have to prove that if $s$ is an integer $\geqslant 1$ such that $H_s(V)=0$, then $V$ is $F$-acyclic. We use induction on $\N(V)$ (see Notation \ref{NV}).

First, if $\N(V)=0$, then $V$ is $F$-acyclic. Next, suppose $\N(V)\geqslant 1$ and $H_s(V)=0$ for some $s\geqslant 1$. By Lemma \ref{acyclicity from V to SV}(i), we have $H_s(SV)=0$. Since $\N(SV)=\N(V)-1$, by induction hypothesis, the $\FI$-module $SV$ is $F$-acyclic. 

Suppose $1\leqslant a \leqslant s$. By Theorem \ref{main theorem}, there are isomorphisms:
\begin{equation*}
H_a(V) \cong SH_{a+1}(V) \cong S^2H_{a+2}(V) \cong \cdots \cong S^{s-a} H_s(V) = 0.
\end{equation*}

Now suppose $a\geqslant s$. By Theorem \ref{main theorem}, there are isomorphisms:
\begin{equation*}
0 = H_s(V) \cong  SH_{s+1}(V) \cong S^2H_{s+2}(V) \cong \cdots \cong S^{a-s} H_a(V),
\end{equation*}
so $H_a(V)_{\mathbf{n}}=0$ for $n\geqslant a-s$. But by Lemma \ref{lemma on lowest degree}, we also have $H_a(V)_{\mathbf{n}}=0$ for $n<a$. Hence, $H_a(V)=0$.

It follows that $V$ is $F$-acyclic.
\end{proof}

A characterization of $F$-acyclicity in terms of existence of a suitable filtration (called $\sharp$-filtration in \cite[Definition 1.10]{Na}) is proved in \cite[Theorem 1.3]{LYu} and \cite[Theorem B]{Ra}; we do not need to use this filtration in our present article.

\subsection{Bounding regularity using the shift functor}
Our second strategy for bounding the Castelnuovo-Mumford regularity $\reg(V)$ of an $\FI$-module $V$ is to find a bound of $\reg(V)$ in terms of $\reg(SV)$, and then use recurrence to obtain a bound for $\reg(V)$. This is similar to the approach used by Li in \cite[Section 4]{Li1}.

\begin{proposition} \label{proposition for bounding reg by S}
Let $V$ be an $\FI$-module. Then
\begin{equation*}
\reg(V) \leqslant \max\{ \hd_1(V)-1,\, \reg(SV)+1 \}.
\end{equation*}
\end{proposition}
\begin{proof}
Set $c=\max\{ \hd_1(V)-1,\, \reg(SV)+1 \}$. There is nothing to prove if $c=\infty$, so assume $c<\infty$.

We shall show, by induction on $a$, that one has:
\begin{equation*}
\hd_a(V)\leqslant c+a \quad \mbox{ for each } a\geqslant 1. 
\end{equation*}
When $a=1$, the inequality is immediate from the definition of $c$. 

Assume that one has $\hd_{a-1}(V)\leqslant c+a-1$ for some $a\geqslant 2$. Then $H_{a-1}(V)_{\mathbf{n}}=0$ for $n\geqslant c+a$. By Theorem \ref{main theorem}, we have an exact sequence:
\begin{equation*}
\cdots \longrightarrow H_a(SV)_{\mathbf{n}} \longrightarrow H_a(V)_{\mathbf{n+1}} \longrightarrow H_{a-1}(V)_{\mathbf{n}} \longrightarrow \cdots.
\end{equation*} 
Since $c\geqslant \reg(SV)+1$, we have $H_a(SV)_{\mathbf{n}}=0$ for $n\geqslant c+a$. Therefore $H_a(V)_{\mathbf{n+1}}=0$ for $n\geqslant c+a$, and hence $\hd_a(V)\leqslant c+a$.
\end{proof}

The following result uses Theorem \ref{acyclic after shifts} to insure the existence of $\N(V)$ (see Notation \ref{NV}).

\begin{theorem} \label{theorem bounding reg by S}
Let $V$ be an $\FI$-module which is generated in degree $\leqslant k$ and related in degree $\leqslant d$ where $k, d\in \Z_+$. Let
\begin{equation*}
\hd_1^S(V) := \max\{ \hd_1(S^i V) + i \mid i=0, 1, \ldots, \N(V) \}.
\end{equation*}
Then 
\begin{equation*}
\reg(V) \leqslant \hd_1^S(V)-1.
\end{equation*}
\end{theorem}
\begin{proof}
We use induction on $\N(V)$. If $\N(V)=0$, then $\reg(V)=-\infty$, so there is nothing to prove. 

Suppose $\N(V)\geqslant 1$. Since $\N(SV)=\N(V)-1$, by induction hypothesis, we have $\reg(SV)\leqslant \hd_1^S(SV)-1$, and hence by Proposition \ref{proposition for bounding reg by S}, we obtain $\reg(V)\leqslant \hd_1^S(V)-1$.
\end{proof}

In the above theorem, one has $\hd_1(S^i V)<\infty$ for each $i$ by Lemma \ref{lemma basic facts}.

It was proved by Li and Ramos \cite[Theorem 5.20]{LRa} that, for a finitely generated $\FI$-module $V$ over a noetherian ring, one has:
\begin{equation*} 
\reg(V) \leqslant \max \{ \deg(H_{\mathfrak{m}}^j(V)) + j \mid j=0, 1,\ldots  \},
\end{equation*}
where $H_{\mathfrak{m}}^j(V)$ for $j=0,1,\ldots $ are the local cohomology groups of $V$. It would be too much of a digression for us to review the definition and properties of local cohomology groups of $\FI$-modules; we refer the reader to the paper \cite{LRa} of Li and Ramos (see  \cite[Definition 5.13 and Theorem E]{LRa}). Let us show that the bound in Theorem \ref{theorem bounding reg by S} is always less than or equal to their bound.

\begin{proposition}
Suppose that $\kk$ is noetherian and $V$ is a finitely generated $\FI$-module over $\kk$. Then one has:
\begin{equation*}
\hd_1^S(V)-1 \leqslant  \max \{ \deg(H_{\mathfrak{m}}^j(V)) + j \mid j=0, 1,\ldots \}.
\end{equation*}
\end{proposition}
\begin{proof}
For any finitely generated $\FI$-module $W$, set 
\begin{equation*}
\varpi(W):=  \max \{ \deg(H_{\mathfrak{m}}^j(W)) + j \mid j=0, 1,\ldots. \}.
\end{equation*}
One has $H_{\mathfrak{m}}^j(SW) \cong SH_{\mathfrak{m}}^j(W)$ for each $j\geqslant 0$ (see \cite[paragraph before Corollary 5.23]{LRa}); thus, one has $\varpi(SW) \leqslant \varpi(W)-1$. Hence, for each $i\in \Z_+$, one has:
\begin{equation*}
\hd_1(S^i V) + i - 1 \leqslant \reg(S^i V) + i \leqslant \varpi(S^i V) + i \leqslant \varpi(V),
\end{equation*}
where the second inequality is obtained by applying \cite[Theorem 5.20]{LRa} to $S^i V$.
\end{proof}

\subsection{Aside on generating degree and relation degree}
Let $V$ be an $\FI$-module which is generated in degree $\leqslant k$ and related in degree $\leqslant d$ for some $k, d\in \Z_+$, that is, there is a short exact sequence 
\begin{equation*}
0\longrightarrow W\longrightarrow P\longrightarrow V\longrightarrow 0
\end{equation*}
where $P$ is a projective $\FI$-module generated in degree $\leqslant k$ and $W$ is an $\FI$-module generated in degree $\leqslant d$. It is easy to see that when such a presentation exists, we have $\hd_0(V)\leqslant k$ and $\hd_1(V)\leqslant d$. Moreover, when such a presentation exists, we can find one with $k=\hd_0(V)$. Whence, suppose that $k=\hd_0(V)$; from the long exact sequence in homology, we obtain \cite[Lemma 4.4]{Li1}:
\begin{equation} \label{gd and hd1}
\hd_1(V) \leqslant \hd_0(W) \leqslant \max\{ \hd_0(V),\, \hd_1(V) \}.
\end{equation}
\begin{lemma} \label{presentation with relation degree the first homological degree}
Let $V$ be an $\FI$-module which is generated in degree $\leqslant k$ and related in degree $\leqslant d$ for some $k, d\in \Z_+$. Suppose that $0\to W\to P\to V\to 0$ is a short exact sequence where $P$ is a projective $\FI$-module generated in degree $\leqslant \hd_0(V)$. If $\hd_0(V)\leqslant \hd_1(V)$, then $W$ is generated in degree $\leqslant \hd_1(V)$.
\end{lemma}
\begin{proof}
Immediate from (\ref{gd and hd1}).
\end{proof}

The following result of Li and Yu \cite[Corollary 3.4]{LYu} says that, where FI-homology is concerned, one can frequently assume that $\hd_0(V)<\hd_1(V)$ and hence apply Lemma \ref{presentation with relation degree the first homological degree}; see \cite[Remark 2.16]{Ra}. Let us give a proof using Proposition \ref{acyclic if one homology is zero}.

\begin{lemma} \label{reduction lemma of Li-Yu}
Let $V$ be an $\FI$-module which is generated in degree $\leqslant k$ and related in degree $\leqslant d$ for some $k, d\in \Z_+$. Suppose that $V$ is not $F$-acyclic. Let $r=\hd_1(V)$ and let $U$ be the $\FI$-submodule of $V$ generated by $\bigsqcup_{n<r} V_{\mathbf{n}}$. Let $W=V/U$. Then one has the followings:
\begin{itemize}
\item[(i)] $W$ is $F$-acyclic.

\item[(ii)] $H_a(U) \cong H_a(V)$ for each $a\geqslant 1$.

\item[(iii)] $\hd_0(U) < \hd_1(U)$.
\end{itemize}
\end{lemma}

\begin{proof}
From the short exact sequence $0\to U\to V\to W\to 0$, we obtain a long exact sequence in homology:
\begin{equation*}
\cdots \longrightarrow H_2(W) \longrightarrow H_1(U) \longrightarrow H_1(V) \longrightarrow H_1(W) \longrightarrow H_0(U) \longrightarrow \cdots.
\end{equation*}
Since $\hd_1(V)=r$ and $\hd_0(U)\leqslant r-1$, we have $\hd_1(W)\leqslant r$. But $\low(W)\geqslant r$, so by Lemma \ref{lemma on lowest degree}, we have $H_1(W)_{\mathbf{n}}=0$ for each $n\leqslant r$. Therefore, we must have $H_1(W)=0$.

By Proposition \ref{acyclic if one homology is zero}, it follows that $W$ is $F$-acyclic. Hence, from the long exact sequence, we see that $H_a(U)\cong H_a(V)$ for each $a\geqslant 1$. In particular, $\hd_1(U)=\hd_1(V)$. Thus, $\hd_0(U)<r=\hd_1(U)$.
\end{proof}

The proof of the above lemma in \cite{LYu} is more elementary than the one we give here. We thought, however, that it might be worthwhile to give a different explanation of why it is true. As observed in \cite{Li3} and \cite{Ra}, one can use Lemma \ref{reduction lemma of Li-Yu} to deduce the following.

\begin{corollary}
Let $V$ be an $\FI$-module which is generated in degree $\leqslant k$ and related in degree $\leqslant d$ for some $k, d\in \Z_+$. Then one has:
\begin{equation*}
\reg(V) \leqslant \min\{ \hd_0(V),\, \hd_1(V) \} + \hd_1(V) -1.
\end{equation*}
If, moreover, $V$ is not $F$-acyclic, then one has:
\begin{equation*}
\N(V) \leqslant \min\{ \hd_0(V),\, \hd_1(V) \} + \hd_1(V).
\end{equation*}
\end{corollary}
\begin{proof}
We may assume that $V$ is not $F$-acyclic. Let $U$ be the $\FI$-submodule of $V$ defined in Lemma \ref{reduction lemma of Li-Yu} and let $W=V/U$. 

By Lemma \ref{presentation with relation degree the first homological degree}, the $\FI$-module $U$ is generated in degree $\leqslant \hd_0(U)$ and related in degree $\leqslant \hd_1(U)$. We have:
\begin{multline*}
\reg(V) = \reg(U) \leqslant \min\{ \hd_0(U),\, \hd_1(U) \} + \hd_1(U) -1 \\ \leqslant \min\{ \hd_0(V),\, \hd_1(V) \} + \hd_1(V) -1,
\end{multline*}
where the first inequality comes from applying Theorem \ref{theorem on bound of reg by D} to $U$. 

Since $W$ is $F$-acyclic, it follows by Lemma \ref{acyclicity from V to SV} that $S^iW$ is $F$-acyclic for each $i\geqslant 0$. From the long exact sequence in homology associated to the short exact sequence $0\to S^iU\to S^iV\to S^iW\to 0$, we deduce that:
\begin{equation*} 
\N(V) = \N(U)\leqslant \min\{ \hd_0(U),\, \hd_1(U) \} + \hd_1(U) \leqslant \min\{ \hd_0(V),\, \hd_1(V) \} + \hd_1(V),
\end{equation*}
where the first inequality comes from applying Theorem \ref{acyclic after shifts} to $U$.
\end{proof}

\subsection{Homological degrees are strictly increasing}
Besides Theorem \ref{main theorem}, the proof of the following result also uses Theorem \ref{acyclic after shifts} to insure the existence of $\N(V)$ (see Notation \ref{NV}).

\begin{theorem} \label{monotonicity theorem}
Let $V$ be an $\FI$-module which is generated in degree $\leqslant k$ and related in degree $\leqslant d$ for some $k, d\in \Z_+$. If $V$ is not $F$-acyclic, then one has:
\begin{equation*}
\hd_1(V) < \hd_2(V) < \hd_3(V) < \cdots.
\end{equation*}
\end{theorem}
\begin{proof}
We use induction on $\N(V)$. Since $V$ is not $F$-acyclic, one has $\N(V)>0$. Moreover, by Proposition \ref{acyclic if one homology is zero}, one has $H_a(V)\neq 0$ for each $a\geqslant 1$. By Theorem \ref{theorem on bound of reg by D}, one has $\hd_a(V)<\infty$ for each $a\geqslant 1$.

Suppose first that $\N(V)=1$. Then $SV$ is $F$-acyclic. From Theorem \ref{main theorem}, one has
\begin{equation*}
 SH_{a+1}(V) \cong H_a(V) \quad \mbox{ for each } a\geqslant 1.
\end{equation*}
This implies $\hd_{a+1}(V) = \hd_a(V)+1$ for each $a\geqslant 1$.

Next, suppose that $\N(V)>1$. Let $a\geqslant 1$ and let $n=\hd_a(V)$. We need to show that $\hd_{a+1}(V)\geqslant n+1$.

Suppose, on the contrary, that $\hd_{a+1}(V)\leqslant n$. Then $H_{a+1}(V)_{\mathbf{n+1}}=0$. From Theorem \ref{main theorem}, we have an exact sequence:
\begin{equation*}
\cdots \longrightarrow H_{a+1}(V)_{\mathbf{n+1}} \longrightarrow H_a(V)_{\mathbf{n}} \longrightarrow H_a(SV)_{\mathbf{n}} \longrightarrow \cdots.
\end{equation*}
Since $H_a(V)_{\mathbf{n}}\neq 0$, it follows that $H_a(SV)_{\mathbf{n}}\neq 0$, and so $\hd_a(SV)\geqslant n$. Since $\N(SV)=\N(V)-1$, by induction hypothesis, one has $\hd_{a+1}(SV)\geqslant n+1$. Thus, there exists $r\geqslant n+1$ such that $H_{a+1}(SV)_{\mathbf{r}}\neq 0$. But from Theorem \ref{main theorem}, we have an exact sequence:
\begin{equation*}
\cdots \longrightarrow H_{a+1}(V)_{\mathbf{r}} \longrightarrow H_{a+1}(SV)_{\mathbf{r}} \longrightarrow H_{a+1}(V)_{\mathbf{r+1}} \longrightarrow \cdots.
\end{equation*}
Since $r>\hd_{a+1}(V)$, we have $H_{a+1}(V)_{\mathbf{r}}=0$ and $H_{a+1}(V)_{\mathbf{r+1}}=0$, so $H_{a+1}(SV)_{\mathbf{r}}=0$, a contradiction. We conclude that $\hd_{a+1}(V) \geqslant n+1$.
\end{proof}

The following corollary uses Theorem \ref{theorem on bound of reg by D} to see that $\reg(V)<\infty$.

\begin{corollary} 
Let $V$ be an $\FI$-module which is generated in degree $\leqslant k$ and related in degree $\leqslant d$ for some $k, d\in \Z_+$. If $V$ is not $F$-acyclic, then there exists $s\geqslant 1$ such that  
\begin{equation*}
\hd_a(V)=\reg(V)+a \quad \mbox{ for each } a \geqslant s.
\end{equation*}
\end{corollary}
\begin{proof}
By Theorem \ref{monotonicity theorem}, we have: 
\begin{equation*}
\hd_1(V)-1\leqslant \hd_2(V)-2 \leqslant \hd_3(V)-3 \leqslant \cdots.
\end{equation*}
But by Theorem \ref{theorem on bound of reg by D}, we have $\reg(V)<\infty$. The claim is now immediate from the definition of $\reg(V)$.
\end{proof}

\end{document}